\documentclass[a4paper,10pt]{article}

\usepackage[utf8]{inputenc}
\usepackage{amsmath}
\usepackage{amssymb}
\usepackage{MnSymbol}
\usepackage{graphicx}
\usepackage{subcaption}
\usepackage{fancyhdr}
\usepackage[shortlabels]{enumitem}
\usepackage[nottoc,numbib]{tocbibind}
\usepackage{xcolor}
\usepackage{aliascnt}
\usepackage{hyperref}
\usepackage{cleveref}
\usepackage{mathtools}
\usepackage{amsfonts}
\usepackage{etoolbox}
\usepackage{amsthm}
\usepackage{pifont}
\usepackage{nicefrac}

\newcommand{\leb}[1]{l^{#1}}
\newcommand{\converges}[2][\infty]{\underset{{#2}\rightarrow{#1}}{\longrightarrow}}
\newcommand{\banachspace}{\mathcal{B}}
\newcommand{\error}{\mathcal{E}}
\DeclarePairedDelimiter\modulus{\lvert}{\rvert}
\newcommand{\Leb}[1]{L^{#1}}
\DeclareMathOperator{\vecspan}{span}
\newcommand{\R}{\mathbb{R}}

\DeclareMathOperator{\dist}{dist}
\newcommand{\N}{\mathbb{N}}
\DeclarePairedDelimiter\norm{\lVert}{\rVert}

\newcommand{\hilbertspace}{\mathcal{H}}
\newcommand{\differential}[1]{\,\mathrm{d}{#1}}
\DeclareMathOperator{\dom}{dom}
\newcommand{\setseparator}{\,:\,}
\DeclareMathOperator{\codim}{codim}

\newcommand{\thetitle}{Approximate Representer Theorems in Non-reflexive Banach Spaces}
\newcommand{\theauthor}{Kevin Schlegel}
\newcommand{\theaddress}{Mathematical Institute\\
University of Oxford\\
Andrew Wiles Building, Radcliffe Observatory Quarter\\
Woodstock Road, Oxford, OX2 6GG, UK}
\newcommand{\theemail}{schlegel@maths.ox.ac.uk}

\newcommand{\citet}[2][]{\cite{#2}{#1}}
\newcommand{\citep}[2][]{\cite{#2}{#1}}

\makeatletter
\newtheoremstyle{theorem}
  {5mm}
  {7mm}
  {\addtolength{\@totalleftmargin}{7mm}
   \addtolength{\linewidth}{-7mm}
   \parshape 1 7mm \linewidth}
  {-7mm}
  {\bfseries}
  {}
  {\newline}
  {\normalfont\textbf{\thmname{#1}\thmnumber{ #2}}\textit{\thmnote{ (#3)}}}

\makeatother
\theoremstyle{theorem}

\newtheorem{theorem}{Theorem}[section]
\AfterEndEnvironment{theorem}{\noindent\ignorespaces}
\crefname{theorem}{theorem}{theorems}
\newaliascnt{proposition}{theorem}
\newtheorem{proposition}[proposition]{Proposition}
\aliascntresetthe{proposition}
\AfterEndEnvironment{proposition}{\noindent\ignorespaces}
\crefname{proposition}{proposition}{propositionss}
\newaliascnt{lemma}{theorem}
\newtheorem{lemma}[lemma]{Lemma}
\aliascntresetthe{lemma}
\AfterEndEnvironment{lemma}{\noindent\ignorespaces}
\crefname{lemma}{lemma}{lemmas}
\newaliascnt{remark}{theorem}

\aliascntresetthe{remark}
\AfterEndEnvironment{remark}{\noindent\ignorespaces}
\crefname{remark}{remark}{remarks}
\newaliascnt{corollary}{theorem}
\newtheorem{corollary}[corollary]{Corollary}
\aliascntresetthe{corollary}
\AfterEndEnvironment{corollary}{\noindent\ignorespaces}
\crefname{corollary}{corollary}{corollaries}
\newaliascnt{definition}{theorem}
\newtheorem{definition}[definition]{Definition}
\aliascntresetthe{definition}
\AfterEndEnvironment{definition}{\noindent\ignorespaces}
\crefname{definition}{definition}{definitions}
\newaliascnt{example}{theorem}

\aliascntresetthe{example}
\AfterEndEnvironment{example}{\noindent\ignorespaces}
\crefname{example}{example}{examples}

\newcommand{\proofend}{\hspace*{\fill}\ding{113}\\}

\let\oldendproof\endproof
\renewcommand{\endproof}{\proofend\oldendproof}

\makeatletter
\newtheoremstyle{notation}
  {5mm}
  {5mm}
  {\addtolength{\@totalleftmargin}{7mm}
   \addtolength{\linewidth}{-7mm}
   \parshape 1 7mm \linewidth}
  {-7mm}
  {\bfseries}
  {}
  {\newline}
  {\normalfont\textbf{\thmname{#1}\thmnumber{ #2}:}\thmnote{ (#3)}}

  \makeatother
\theoremstyle{notation}

\crefname{notation}{notation}{notations}

\newtheoremstyle{subproof}
  {0mm}
  {5mm}
  {}
  {}
  {\bfseries}
  {}
  {\newline}
  {\normalfont\textit{\textbf{\thmname{#1}\thmnumber{ #2}:}\thmnote{ (#3)}}}

\theoremstyle{subproof}

\AfterEndEnvironment{subproof}{\noindent\ignorespaces}
\crefname{subproof}{part}{parts}

\hypersetup{
  colorlinks=false,
  linkbordercolor=red,
  pdfborderstyle={/S/U/W .5}
}

\fancypagestyle{paper}{
  \fancyhead{}
  \fancyhead[L]{\footnotesize\nouppercase{\leftmark}}
  \fancyhead[R]{\footnotesize\theauthor}
  \fancyfoot{}
  \fancyfoot[C]{\thepage}
}

\renewcommand{\maketitle}{
  \thispagestyle{empty}
  \begin{center}
    {\Large \thetitle \par}
    \vspace{5mm}
    {\large \theauthor \par}
    \vspace{2mm}
    {\small \theaddress \par}
    {\small Email: \textit{\theemail} \par}
    \vspace{3mm}
    \today
    \vspace{8mm}
  \end{center}
}
\begin{document}

\maketitle

\pagestyle{paper}
{\Large\bf Abstract}\\
The representer theorem is one of the most important mathematical foundations for regularised learning and kernel methods. Classical formulations of the theorem state sufficient conditions under which a regularisation problem on a Hilbert space admits a solution in the subspace spanned by the representers of the data points. This turns the problem into an equivalent optimisation problem in a finite dimensional space, making it computationally tractable. Moreover, Banach space methods for learning have been receiving more and more attention. Considering the representer theorem in Banach spaces is hence of increasing importance. Recently the question of the necessary condition for a representer theorem to hold in Hilbert spaces and certain Banach spaces has been considered. It has been shown that a classical representer theorem cannot exist in general in non-reflexive Banach spaces. In this paper we propose a notion of approximate solutions and approximate representer theorem to overcome this problem. We show that for these notions we can indeed extend the previous results to obtain a unified theory for the existence of representer theorems in any general Banach spaces, in particular including $\leb{1}$-type spaces. We give a precise characterisation when a regulariser admits a classical representer theorem and when only an approximate representer theorem is possible.


\section{Introduction}
It is a common approach in learning theory to formulate a problem of estimating functions from input and output data as an optimisation problem. Most commonly used is regularisation, in particular \textit{Tikhonov regularisation} where we consider an optimisation problem of the form
\[\min\left\{\mathcal{E}({(\langle f,x_i\rangle,y_i)}^m_{i=1}) + \lambda\Omega(f)\setseparator f\in\hilbertspace\right\}\]
where $\hilbertspace$ is a Hilbert space $\langle\cdot,\cdot\rangle$, $\left\{(x_i,y_i) \setseparator i=1,\ldots,m\right\}\subset\hilbertspace\times Y$ is a set of given input/output data with $Y\subseteq\R$, $\mathcal{E}\colon\R^m\times Y^m\rightarrow\R$ is an \textit{error function}, $\Omega\,\colon\hilbertspace\rightarrow\R$ a \textit{regulariser} and $\lambda>0$ is a \textit{regularisation parameter}. The representer theorem is one of the most important mathematical foundations for such regularised learning problems. It states that under certain conditions on the regulariser the optimisation problem has a solution in the finite dimensional subspace spanned by the data points $x_i\in\hilbertspace$, making it computationally tractable.\\
While these problems are well understood in Hilbert spaces, Banach space methods have been receiving more and more attention in machine learning for various reasons, such as e.g.\ the richer geometric variety in comparison to Hilbert spaces, and certain desirable properties of Banach space norms such as the $\leb{1}$ norm inducing sparsity of the solution vector. We are thus going to consider the more general regularisation problem
\begin{equation}
  \label{eq:regularisation_problem}
  \inf\left\{\error({(L_i(f),y_i)}_{i=1}^m) + \lambda\Omega(f)\setseparator f\in\banachspace\right\}
\end{equation}
where $\banachspace$ is a Banach space and the $L_i$ are continuous linear functionals on $\banachspace$. This framework is general enough to include all classical Hilbert space techniques such as least squares, SVMs and Kernel PCA but also their counterparts in reproducing Kernel Banach spaces introduced by Zhang, Xu and Zhang~\citet{zhang:09,zhang:12}. Furthermore it includes popular regularisation frameworks such as lasso \citep{tibshirani:96} and its variants, e.g.\ square-root lasso \citep{belloni:11}.\\
Moreover, while the $L_i$ could be simple point evaluations $L_i(f) = f(x_i)$, phrasing the problem using general linear functionals has the advantage of including other interesting cases such as local averages of the form $L(f) = \int_{\banachspace}f(x)\differential{P(x)}$ where $P$ is a probability measure on $\banachspace$.\\
With the data given as functionals in the dual space $\banachspace^\ast$ it is clear that the representer theorem in Banach spaces in fact has to be rooted in the dual space rather than the space itself, as can also be seen in the work by Micchelli and Pontil and Zhang, Xu and Zhang~\citet{miccelli:04,zhang:09,zhang:12} and our earlier work~\citep{schlegel:19,schlegel:19.2}. It turns out that the representer theorem is closely related to the properties of the duality mapping
\begin{equation}\label{eq:duality_mapping}
  J:\banachspace \rightarrow 2^{\banachspace^\ast} \qquad J(f) = \left\{L\in \banachspace^\ast \setseparator L(f) = \norm{L}\cdot\norm{f}, \norm{L}=\norm{f} \right\}
\end{equation}
This does not become apparent in Hilbert spaces as the duality mapping is the identity. Before we discuss this in more detail we introduce another common assumption to simplify the problem. While in applications we are often interested in regularisation problems of the form~(\ref{eq:regularisation_problem}), Argyriou, Micchelli and Pontil~\citet{argyriou:09} and our earlier work \citep{schlegel:19,schlegel:19.2} show that in Hilbert spaces and reflexive Banach spaces under very mild conditions~(\ref{eq:regularisation_problem}) admits a representer theorem if and only if the regularised interpolation problem
\begin{equation}
\label{eq:interpolation_problem}
  \inf\left\{\Omega(f)\setseparator f\in\banachspace, L_i(f)=y_i\,\forall i=1,\ldots,m\right\}
\end{equation}
admits a representer theorem. Here by admitting a representer theorem we mean that a solution determined by a linear combination of the data always exists whenever the constraints can be satisfied. In this case we will call $\Omega$ admissible. The connection between regularisation and regularised interpolation is not surprising as the regularisation problem is more general and one obtains a regularised interpolation problem in the limit as the regularisation parameter goes to zero. Thus we can, and will, focus our attention on the regularised interpolation problem which is more convenient to study. The precise statement of this fact with the required conditions and its proof for general Banach spaces are presented in \cref{sec:regularisation-interpolation}, as the proof only requires a few technical modifications from the one presented in our previous work~\citep{schlegel:19.2}. Note that in fact any representer theorem for regularised interpolation holds for any regularisation problem with the same regulariser without any further assumptions. Thus any representer theorem for regularised interpolation proved below is immediately valid for regularisation problems of the form~(\ref{eq:regularisation_problem}).\\
\ \\
It is well known that a regulariser is admissible if it is a nondecreasing function of the Hilbert space norm. By a Hahn-Banach argument as e.g.\ by Zhang and Zhang~\citet{zhang:12} the same is true for reflexive Banach spaces. Argyriou, Micchelli and Pontil~\citet{argyriou:09} showed that this condition is also necessary for differentiable regularisers on Hilbert spaces. Dinuzzo and Sch\"{o}lkopf~\citet{dinuzzo:12} extend this result to lower semicontinuous regularisers on Hilbert spaces. Recently we removed the regularity assumptions on the regulariser \citep{schlegel:19}, proving that an admissible regulariser cannot be very far from being a nondecreasing function of the norm, in a sense made precise in the paper. Moreover the results apply to uniformly convex, uniformly smooth Banach spaces, extending the theory to a wide range of Banach spaces. More recently we further showed that in fact the same necessary and sufficient condition holds for reflexive Banach spaces \citep{schlegel:19.2}. It is interesting, and instructive for this work, to note that our previous work clearly highlights the relationship between the properties of the duality mapping~(\ref{eq:duality_mapping}) and the formulation of the representer theorem. To account for the nonlinearity of the duality mapping in uniform Banach spaces \citep{schlegel:19} we defined a regulariser to be admissible if there exists a solution $f_0$ to~(\ref{eq:interpolation_problem}) with dual element in the linear span of the linear functionals defining the interpolation problem, i.e. $\sum c_i L_i = J(f_0)$. To account for the duality mapping not being univocal in Banach spaces which are not smooth \citep{schlegel:19.2} this equality turns into an inclusion, i.e. $\sum c_i L_i \in J(f_0)$.\\
Moreover, by giving a counterexample \citep{schlegel:19.2} we showed that it is not possible in general to obtain a representer theorem in this sense if the space is not reflexive. This is unfortunate since $\leb{1}$, which is frequently used in applications, is not reflexive. Only the finite dimensional $\leb{1}_n$ is reflexive.\\
\ \\
To overcome this issue we propose to follow the approach of reflecting the properties of the duality mapping in the formulation of the representer theorem. The reason why a representer theorem in the above sense cannot exist in a non-reflexive Banach space is that the duality mapping is not surjective. This means that we cannot expect to find a solution with dual element in the linear span of the linear functionals defining the optimisation problem as described above. But Bishop and Phelps~\citet{bishop:61} prove that every Banach space is subreflexive, i.e.\ the image of the duality mapping $J$ is norm-dense in $\banachspace^\ast$. Thus we can hope to be able to get arbitrarily close to $\vecspan\{L_i\}$, i.e.\ $\dist(J(f_0),\vecspan\{L_i\})<\varepsilon$.\ This leads to a notion of \textit{approximate solution} and \textit{approximate representer theorem} which we are going to introduce in this paper. We are going to show that for this weaker concept of solutions we can indeed obtain the immediate generalisations of the results of Argyriou, Micchelli and Pontil~\citet{argyriou:09} and our earlier work \citep{schlegel:19,schlegel:19.2}. This provides a unified theory for the existence of representer theorems in arbitrary Banach spaces, in particular including $\leb{1}$-type spaces which are very frequently used in applications.

\section{Approximate representer theorems}\label{sec:representer_theorems}
We let $\banachspace$ be an arbitrary Banach space with duality mapping~(\ref{eq:duality_mapping}) and consider the regularised interpolation problem~(\ref{eq:interpolation_problem}). There are two main differences to the setting of reflexive Banach spaces that need to be overcome.\\
Firstly, Argyriou, Micchelli and Pontil~\citet{argyriou:09} and our earlier work \citep{schlegel:19,schlegel:19.2} assume that a minimiser of~(\ref{eq:interpolation_problem}) always exists, whenever the constraints can be satisfied. But in a non-reflexive Banach space we cannot expect the minimum of~(\ref{eq:interpolation_problem}) to always be attained. More precisely, if we denote by $Z$ the subspace
  \[Z = \bigcap\,_{i=1}^m\ker(L_i)\]
  it is easy to see that solving the minimal norm interpolation problem, i.e.\ the case
  $\Omega(f)=\norm{f}_\banachspace$ in~(\ref{eq:interpolation_problem}), is equivalent to minimising $\inf\{\norm{\overline{f}+f_T}_\banachspace \setseparator f_T\in Z\}$ where $\overline{f}\in\banachspace$ is any function satisfying the interpolation constraints. In other words the infimum of the minimal norm interpolation is attained at $f_0$ if and only if the distance of $0$ to the affine space $\overline{f}+Z$ is attained at $f_0\in\overline{f}+Z$. It is well known that such $f_0$ does not always exist if $\banachspace$ is not reflexive. Now different values of the $y_i$ correspond to different shifts $\overline{f}$ of $Z$ so that if the distance is attained, it happens at different points. Thus a solution to the minimal norm interpolation always exist for any given data exactly when $Z$ is proximinal.
\begin{definition}[Proximinal Subspace]
  Let $V$ be a real normed vector space and $W\subset V$ a closed subspace of $V$. We say $W$ is proximinal if the distance from any point in $V$ to $W$ is attained, i.e.\ for every $x\in V$ there is a $y\in W$ such that $\norm{x-y}_V = \dist(x,W)$.
\end{definition}
Following this intuition, instead of assuming a solution to the regularised interpolation always exists when the constraints can be satisfied, we will assume that a solution to \cref{eq:interpolation_problem} always exists if $Z$ is proximinal. While in a reflexive space every closed linear subspace is proximinal the question becomes a lot more delicate in non-reflexive spaces and there are spaces which contain in a sense very few proximinal subspaces, e.g.\ no proximinal subspace of finite codimension greater than one~\citep{read:18,kadets:18}.
Conditions for when a subspace is proximinal are still an active area of research. Some good references for what is known include the books by Singer, Holmes and Conway~\citet{singer:70,holmes:75,conway:94}. We state two results which are of particular relevance to our work in \cref{sec:proximinal-subspaces}.\\
\ \\
Secondly the duality mapping $J$ is surjective if and only if the space is reflexive. Thus $\vecspan{L_i}$ may not be entirely contained in the image of $J$, or as we illustrate in our earlier work \citep{schlegel:19.2}, possibly even $J(\banachspace) \cap \vecspan\{L_i\} = \emptyset$. We thus cannot hope for a solution with a dual element in the linear span of the functionals, i.e. $J(f_0)\cap\vecspan\{L_i\}\neq\emptyset$. But since every Banach space is subreflexive~\citep{bishop:61}, which means the image of the duality mapping is norm dense in the dual space, we might expect to be able to get arbitrarily close to the linear span, i.e. $\dist(J(f_0),\vecspan\{L_i\}) < \varepsilon$.\\
Combining both, approximation of the infimum in~(\ref{eq:interpolation_problem}) and norm-closeness to the span of the $L_i$ leads to the afore mentioned notion of \textit{approximate solution} and \textit{approximate representer theorem} and hence a new definition of admissibility of regularisers.
\begin{definition}[Admissible Regularizer]\label{def:admissibility}
We say a function $\Omega\,\colon \banachspace\rightarrow\R$ is admissible if for any $m\in\N$ and any given data $\{L_1,\ldots,L_m\}\subset\banachspace^\ast$ and $\{y_1,\ldots,y_m\}\subset Y$ such that the interpolation constraints can be satisfied the regularised interpolation problem \cref{eq:interpolation_problem} either
\begin{enumerate}[(i)]
\item\label{item:admissible_i} Admits a solution $f_0$ such that there exist coefficients  $\{c_1,\ldots,c_m\}\subset\R$ such that
  \[\hat{L}=\sum\limits_{i=1}^m c_i L_i \in J(f_0) \qquad \mbox{if }Z = \bigcap_{i\in\N_m}\ker(L_i) \mbox{ is proximinal}\]

\item\label{item:admissible_ii} Or otherwise admits for every $\varepsilon>0$ an approximate solution $f_0^\varepsilon$ such that
  \[\Omega(f_0^\varepsilon) \leq \inf\left\{\Omega(f) \setseparator f\in\banachspace,\, L_i(f)=y_i\,\forall i=1,\ldots,m\right\} + \varepsilon\]
  and there exist $\hat{L}\in J(f_0^\varepsilon)$ and coefficients $\{c_1,\ldots,c_m\}\subset\R$ such that
  \[\norm{\hat{L} - \sum\limits_{i=1}^m c_i L_i}_{\banachspace^\ast} < \varepsilon\]
\end{enumerate}
\end{definition}

\subsection{Existence of approximate representer theorems}
We now show that with this notion of admissibility we can indeed obtain the analogue of the results of Argyriou, Micchelli and Pontil~\citet{argyriou:09} and our previous work \citep{schlegel:19,schlegel:19.2} that being in a sense nondecreasing along tangents is a necessary and sufficient condition for admissibility. As became apparent in the case of reflexive Banach spaces \citep{schlegel:19.2}, when the space is not strictly convex we can only hope to characterise the regulariser as a function of the faces of the norm ball. Recall that an exposed face $F$ of the norm ball $B_r\subset\banachspace$ is a non-empty subset of $B_r$ such that $F = \left\{x\in B_r \setseparator L(x) = \sup_{y\in B_r} L(y)\right\}$ for some $L\in\banachspace^\ast$ (for more details see e.g. \citet{hiriart:04,aizpuru:05}).
\begin{lemma}\label{lma:tangential_bound}
  A function $\Omega\,\colon\banachspace\rightarrow\R$ is admissible if and only if for every exposed face of the norm ball, $\Omega$ attains its minimum in at least one point and for every $f$ in the face where the minimum is attained and every $L\in J(f)$ exposing the face and every $f_T\in\ker(L)$ we have
  \[\Omega(f+f_T) \geq \Omega(f)\]
\end{lemma}
\begin{definition}
  We are going to refer to the points \cref{lma:tangential_bound} applies to as \textit{admissible points}.
\end{definition}
\newpage
\begin{proof}\\
  {\bf Part 1: }\textit{$\Omega$ admissible $\Rightarrow$ nondecreasing along tangential directions}\\
  Fix any $f\in\banachspace$ and consider, for $L\in J(f)$ arbitrary but fixed, the regularised interpolation problem
    \[\min\left\{\Omega(g) \setseparator g\in\banachspace, L(g)=L(f)=\norm{f}^2\right\}\]
    Conway~(\cite{conway:94}~Prop.~4.7) proves that $\ker(L)$ is proximinal if and only if $L$ is in the image of the duality mapping. As $\Omega$ is assumed to be admissible we thus are in the case~\ref{item:admissible_i} of \cref{def:admissibility} and there exists a solution $f_0$ such that $c\cdot L\in J(f_0)$. We can thus argue exactly as in the case of a reflexive space, we include the short proof for completeness.\\
    If there does not exist $g\in\banachspace$ such that $g\neq f$ and $L\in J(g)$ then the solution can only be $f$ itself. Then for any $f_T\in\ker(L)$ also $L(f+f_T)=L(f)=\norm{f}^2$ and $f+f_T$ also satisfies the constraints and hence necessarily $\Omega(f+f_T) \geq \Omega(f)$.\\
    But if there exists $f\neq g\in\banachspace$ such that $L\in J(g)$ we have no way of making a statement about how $\Omega(f)$ and $\Omega(g)$ compare. All we can say is that in this face there is at least one point where the minimum of $\Omega$ is attained. It is clear that for any of those minimal points the above discussion is true for $L$ exposing the face so that we obtain the tangential bound.\\
    \ \\
    {\bf Part 2: }\textit{Nondecreasing along tangential directions $\Rightarrow$ $\Omega$ admissible}\\
  Fix any data $(L_i,y_i)\in\banachspace^\ast\times Y$ for $i=1,\ldots,m$ such that the constraints can be satisfied. We now have the two cases of \cref{def:admissibility} to consider.\\
  \textbf{Case 1:} If $Z$ is proximinal then by assumption there exists a solution $f_0$ of the regularised interpolation problem and we are looking for a solution in the sense of \cref{def:admissibility}~\ref{item:admissible_i}. We need to show that if $f_0$ is not a solution in this sense then there exists $f_T\in Z$ such that $\vecspan\{L_i\}\cap J(f_0+f_T) \neq \emptyset$. It turns out that the proof for reflexive Banach spaces~\citep{schlegel:19.2} remains valid, and understanding its main ideas is instructive for dealing with the second case. The proof is based on minimising the functional
  \begin{equation}\label{eq:norm_duality_functional}
    F_{f_0}\colon\banachspace\rightarrow\R, \quad F_{f_0}(f) = \int\limits_{0}^{\norm{f-f_0}} t\differential{t} = \frac{\norm{f-f_0}^2}{2}
  \end{equation}
  over the subspace $Z$. Reflexivity of $\banachspace$ is only used to ensure reflexivity of  $Z$ and thus the existence of a minimiser on $Z$ of the continuous, convex and coercive functional $F_{f_0}$. But this minimiser clearly exists exactly when the metric projection of $f_0$ onto $Z$ exists, thus by definition when $Z$ is proximinal. One can check that with the existence of a minimiser of $F_{f_0}$ on $Z$ the rest of the proof for reflexive spaces remains valid. Again we include the remaining short argument for completeness.\\
  For the minimiser $f_T\in Z$ of $F_{f_0}$ we have that there exists $L\in J(f_0+f_T)$ such that $L\big|_Z\equiv0$. Since $\vecspan\{L_i\}=Z^\perp$ this in turn means that $L\in\vecspan\{L_i\}$. It remains to show that $\hat{f}$ indeed minimises $\Omega$. But for $L\in J(f_0+f_T)\cap Z^\perp$ we have $-f_T\in\ker(L)$. If $f_0+f_T$ is exposed by $L$ then the tangential bound applies and
  \[\Omega(f_0+f_T) \leq \Omega((f_0 + f_T)  + (-f_T)) = \Omega(f_0)\]
  so $f_0+f_T$ is a solution of the regularised interpolation problem.\\
  If on the other hand $f_0+f_T$ is not exposed by $L$, then it is contained in a face exposed by $L$. But then for any $\overline{f_T}\in\banachspace$ such that $(f_0+f_T)+\overline{f_T}$ is still contained in this face we have that $L\in J(f_0+f_T+\overline{f_T})$ and $\overline{f_T}\in\ker(L)$ so that $f_0+f_T+\overline{f_T}$ satisfies the interpolation constraints. We can thus choose $\overline{f_T}$ such that $f_0+f_T+\overline{f_T}$ is a minimum of $\Omega$ in the face and the tangential bound applies to it. Thus similarly to before
  \[\Omega(f_0+f_T+\overline{f_T}) \leq \Omega((f_0 + f_T+\overline{f_T})  + (-f_T-\overline{f_T})) = \Omega(f_0)\]
  and $f_0+f_T+\overline{f_T}$ is a solution of the regularised interpolation problem of the desired form.\\
  \ \\
  \textbf{Case 2:} If $Z$ is not proximinal the existence of a minimiser of~(\ref{eq:interpolation_problem}) is not guaranteed. But for every $\varepsilon>0$ there exists $f_0^\varepsilon$ which $\varepsilon$-almost attains the infimum. We need to show that if any such $f_0^\varepsilon$ is not a solution in the sense of \cref{def:admissibility}~\ref{item:admissible_ii} then there exists $f_T^\varepsilon\in Z$ such that $f_0^\varepsilon+f_T^\varepsilon$ is, i.e.\ $\dist(J(f_0^\varepsilon+f_T^\varepsilon),\vecspan\{L_i\}) < \varepsilon$.\\
  Following the approach from case 1 this means we are looking for $f_T^\varepsilon$ with $L\in J(f_0^\varepsilon+f_T^\varepsilon)$ such that $\norm{L\big|_Z}<\varepsilon$. We are again going to consider the functional $F_{f_0^\varepsilon}$ as defined in~(\ref{eq:norm_duality_functional}), for simplicity denoted by $F$ below. With $Z$ not proximinal we do not get a minimiser of $F\big|_Z$ anymore. But by Ekelands variational principle~\citep{ekeland:74} for every $\varepsilon>0$ there exists an approximate minimiser $f_T^\varepsilon \in Z$ such that
    \begin{equation}
      \label{eq:ekeland}
      F(f_T^\varepsilon) \leq \inf_{f\in Z}F(f) + \varepsilon
      \quad\mbox{and}\quad
      F(f_T^\varepsilon) - F(g) < \varepsilon\cdot\norm{f_T^\varepsilon-g} \quad \forall f_T^\varepsilon\neq g\in Z
    \end{equation}
    Choosing $g=f_T^\varepsilon+th$ for $h\in Z$ in \cref{eq:ekeland} we obtain a bound on the directional derivative of $F$

    \begin{equation}
      \label{eq:directional_derivative_bound}
      F'(f_T^\varepsilon,h) = \lim\limits_{t\searrow 0}\frac{F(f_T^\varepsilon+th) - F(f_T^\varepsilon)}{t} > -\varepsilon\cdot\norm{h}
    \end{equation}
    By a corollary of the Sandwich theorem by Simons (\Cref{sec:sandwich-theorem} \cref{clry:sandwich_thm_subdif}) there exists $L\in Z^\ast$ such that $L\in\partial F\big|_Z(f_T^\varepsilon)$ which is necessary to extend it to $L\in J(f_0^\varepsilon+f_T^\varepsilon)$. Moreover
    \[\inf\limits_{h\in B}L(h) = \inf\limits_{h\in B}F'(f_T^\varepsilon,h) \overset{\ref{eq:directional_derivative_bound}}{>} -\varepsilon\cdot\norm{h}\]
    which implies that $\norm{L}_{Z^\ast}< \varepsilon$. By a Hahn-Banach argument this functional can be extended to an $L\in J(f_0^\varepsilon+f_T^\varepsilon)$ such that $\dist(L,\vecspan\{L_i\}) < \varepsilon$. The construction is not difficult but technical and given in \cref{sec:extend_functional}. Thus $f_0^\varepsilon+f_T^\varepsilon$ satisfies the assumptions of \cref{def:admissibility}~\ref{item:admissible_ii}\hspace*{0.1mm}.\\
    The fact that $f_0^\varepsilon+f_T^\varepsilon$ indeed minimises $\Omega$ follows in the same way as in case 1. If $f_0^\varepsilon+f_T^\varepsilon$ is an exposed point it satisfies the tangential bound and thus
$$\Omega(f_0^\varepsilon+f_T^\varepsilon) \leq \Omega((f_0^\varepsilon + f_T^\varepsilon) + (-f_T^\varepsilon)) = \Omega(f_0^\varepsilon)$$
If $f_0^\varepsilon+f_T^\varepsilon$ is not exposed it is contained in a face and just as before we can add another $\overline{f_T}\in Z$ so that the sum is within the face and
$$\Omega(f_0^\varepsilon+f_T^\varepsilon+\overline{f_T}) \leq \Omega((f_0^\varepsilon + f_T^\varepsilon + \overline{f_T}) + (-f_T^\varepsilon - \overline{f_T})) = \Omega(f_0^\varepsilon)$$
Since this new point is in the same face it has the same $L$ as a dual element and is thus an admissible solution.
\end{proof}

\subsection{Uniformly non-rotund spaces}\label{sec:uniformly-non-rotund}
Argyriou, Micchelli and Pontil~\citet{argyriou:09} and our earlier work~\citep{schlegel:19,schlegel:19.2} prove that being tangentially nondecreasing is equivalent to being (almost) radially symmetric. We now want to prove the corresponding geometric interpretation of \cref{lma:tangential_bound}. As we argued in the case of reflexive Banach spaces~\citep{schlegel:19.2}, the geometric variety of arbitrary Banach spaces does not allow for a general, closed form result of this kind. It is clear that our arguments for strictly convex spaces remain true even without reflexivity, but with the most important examples of non-reflexive spaces being $\leb{1}$ and $\Leb{1}$ we are going to introduce and consider a class of function spaces which in particular contains those spaces. The results we obtain are closely related to the ones for $\leb{1}_n$ in~\citep{schlegel:19.2}.\\
Recall that a point $x\in\banachspace$ is rotund if for any $y\in\banachspace$ such that $\norm{y}=\norm{x}$ we have $\norm{y} = \norm{\frac{x+y}{2}}$ implies $x=y$.
\begin{definition}[Uniformly non-rotundness]
  We say a point $0\neq f\in\banachspace$ is uniformly non-rotund if it is not rotund for any two dimensional subspace of $\banachspace$ containing it. In other words, $f$ is not rotund in any direction. We say the space $\banachspace$ is uniformly non-rotund if every $0\neq f\in \banachspace$ is uniformly non-rotund.
\end{definition}
The main reason for uniform non-rotundness to be useful is because it means that there cannot exist faces with a smooth boundary. If any part of the boundary of a face was smooth one would be able to find a two dimensional subspace containing the smooth boundary point and a rotund point in its neighbourhood. If no point in the boundary of a face is smooth then the boundary consists of faces of a lower dimension. As the faces are closed convex sets forming the surface of the norm ball this means that the boundary of a face is given by the intersections with its neighbouring faces. These lower dimensional faces are exposed by another functional and contain their own minimum of $\Omega$. This provides us with a way of running a similar argument as in the cases of uniform and reflexive Banach spaces~\citep{schlegel:19,schlegel:19.2}. From any admissible point we can reach a minimum on the boundary of its face and from there either go back for a radial bound or move further around the ball for a circular bound.
\begin{lemma}\label{lma:circular_bound}
  If for every exposed face of the ball $\Omega$ attains its minimum in at least one point, and for every $f$ in the face where the minimum is attained and every $L\in J(f)$ exposing the face and every $f_T\in\ker(L)$ we have $\Omega(f+f_T) \geq \Omega(f)$, then for any fixed admissible $\hat{f}\in\banachspace$ we have that
  \[\Omega(\hat{f})\leq\Omega(f)\]
  for all $f\in\banachspace$ such that $\norm{\hat{f}} < \norm{f}$.
\end{lemma}
\begin{proof}
  Once again we follow the proof ideas as for reflexive Banach spaces~\citep{schlegel:19.2}. In particular the proof for $\leb{1}_n$ is instructive. More precisely, the tangential bound from \cref{lma:tangential_bound} can be extended to a radial bound by moving ``out and back'' along tangents. But since the minimum can occur anywhere within the face we actually view $\Omega$ as a function $\overline{\Omega}$ of the faces $F$ of the norm ball in $\banachspace$
  \[\overline{\Omega}(F) = \min\limits_{f\in F}\Omega(f)\]
  We are going to prove that $\overline{\Omega}$ is monotone along the ray $\lambda F$, $\lambda>1$, i.e.\ the minimum of $\Omega$ within a face is nondecreasing as a function of the norm. Since each minimum satisfies the tangential bound this gives the half space bound for all half spaces defined by a tangent plane through the minimum $\hat{f}$, given by some $\hat{L}\in J(\hat{f})$, as illustrated in \cref{fig:radially_increasing_non_rotund_minimum,fig:radially_increasing_non_rotund_exposed}\hspace*{0.2mm}. Moreover by repeatedly moving along tangents we can extend the tangential bound all the way around the circle as can be seen in \cref{fig:radiating_bound_non_rotund}.\\
  But since a general Banach space may not contain any exposed points we need to be more careful than in the cases of strictly convex Banach spaces and $\leb{1}_n$. The difficulties lie in the fact that we need to prove for both arguments that we can always find \textit{admissible points} at which to consider the tangents.\\
  \begin{figure}[h]
      \begin{subfigure}[t]{.32\textwidth}
        \centering
        \includegraphics[width=.95\linewidth]{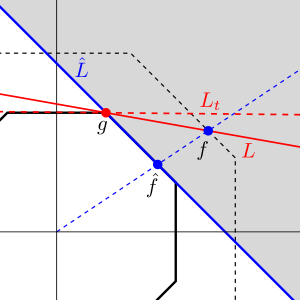}
        \caption{\footnotesize If $\hat{f}$ was the minimum in the face $\hat{F}$, then it has the tangential bound from $\hat{L}$ to reach $\overline{g}$. From $\overline{g}$ we have the tangential bound from $L_t$ to reach any point within $\lambda\hat{F}$ for $1<\lambda<1+\varepsilon$, in particular the minimum within the face.}\label{fig:radially_increasing_non_rotund_minimum}
      \end{subfigure}
      \begin{subfigure}[t]{.32\textwidth}
        \centering
        \includegraphics[width=.95\linewidth]{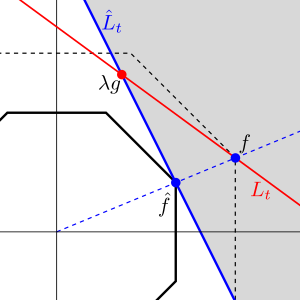}
        \caption{\footnotesize If $\hat{f}$ was an exposed point, then we can construct a set of functionals $\hat{L}_t$ which expose $\hat{f}$ and hit $\lambda\overline{g}$, the minimum in the face $\lambda F_t$. For $\lambda\overline{g}$ we then get a tangential bound back to the face $\mu\hat{F}$.}\label{fig:radially_increasing_non_rotund_exposed}
      \end{subfigure}
      \begin{subfigure}[t]{.32\textwidth}
        \centering
        \includegraphics[width=.95\linewidth]{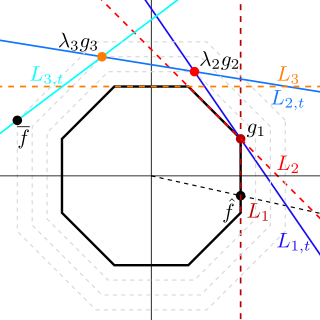}
        \caption{\footnotesize We can move around the circle along points which are exposed in the two dimensional subspace, while staying arbitrarily close to the circle.}\label{fig:radiating_bound_non_rotund}
      \end{subfigure}
      \caption{The tangential bound can be extended radially and around the ball.}\label{fig:radially_increasing_non_rotund}
    \end{figure}

    \noindent
    \textbf{Part 1: }\textit{Bound $\Omega$ on the half spaces given by the tangent planes through $\hat{f}$}\\
    We start by proving that $\overline{\Omega}$ is radially nondecreasing. Note that we don't need to show monotonicity for the entire ray $\lambda F$ for $1<\lambda$. It is sufficient to consider $1<\lambda<1+\varepsilon$ as long as the $\varepsilon$ is at least nondecreasing as a function of the norm along the ray.\\
    Fix an admissible $\hat{f}\in\banachspace$ and let $X$ be any 2-dimensional subspace containing $\hat{f}$. As $\banachspace$ is uniformly non-rotund no point in $X$ is rotund so its unit ball consists of straight line sections and corners as shown in \cref{fig:radially_increasing_non_rotund}. In particular there exists $g\neq\hat{f}$ in the same straight section as $\hat{f}$ and exposed in $X$. It is also clear that there are linear functionals $\hat{L},L\in X^\ast$, where $\hat{L}$ exposes the straight segment containing $\hat{f}$ and $g$, and $L$ exposes only the point $g$. By Hahn-Banach there are extensions of these functionals to $\banachspace$, also denoted by $\hat{L}$ and $L$, exposing faces $\hat{F}$ and $F$ respectively.\\
    We now let $L_t = t\hat{L} + (1-t)L$, $t\in(0,1)$ so that $L_t$ exposes the face $F_t = \hat{F}\bigcap F$ which is strictly smaller than $\hat{F}$. Thus $\Omega$ has a minimum in $F_t$, $\overline{g}$ say. Since $\overline{g}\in F_t \subset \hat{F}$ it is clear that $\hat{L}$ attains its norm at $\overline{g}$ which means that there is a tangent from $\hat{f}$ to $\overline{g}$. Being the minimum in $F_t$ we have that $\overline{g}$ has the tangential bound for all $L_t$.\\
    Putting those observations together we obtain the claimed bound. If $\hat{f}$ was the minimum in the face $\hat{F}$, then it has the tangential bound from $\hat{L}$ to reach $\overline{g}$. From $\overline{g}$ we have the tangential bound from $L_t$ to reach any point within $\lambda\hat{F}$ for $1<\lambda<1+\varepsilon$, in particular the minimum within the face. This is illustrated in \cref{fig:radially_increasing_non_rotund_minimum}.\\
    If on the other hand $\hat{f}$ was an exposed point, then it is clear that using an argument similar to the one above we can construct a set of functionals $\hat{L}_t$ which expose $\hat{f}$ and hit $\lambda\overline{g}$, the minimum in the face $\lambda F_t$. For $\lambda\overline{g}$ we then get a tangential bound back to the face containing $\mu\hat{f}$ in the same way as above. This is illustrated in \cref{fig:radially_increasing_non_rotund_exposed}.\\
    This shows that the minimum of $\Omega$ for any fixed face $F$ is indeed monotone, which in turn means that any admissible point bounds every point in the open half spaces spanned by a tangent plane at the point.\\
    \ \\
    \textbf{Part 2: }\textit{Extend the bound around the circle}\\
    Next we show that from any fixed admissible point $\hat{f}$ we can reach every other admissible point of norm strictly bigger than $\norm{\hat{f}}$. This combined with the half space bound gives the claimed bound for all points outside the circle.\\
    Fix an admissible point $\hat{f}\in\banachspace$ and the admissible point $\overline{f}\neq\hat{f}$ with $\norm{\overline{f}}>\norm{\hat{f}}$ to be reached. Then $\hat{f}$ and $\overline{f}$ span a two dimensional subspace $X$. As before $X$ only consists of straight line sections and corners. Clearly we can construct a sequence of points $g_i$ and linear functionals $L_i$ exposing the straight line section from $g_{i-1}$ to $g_i$ as illustrated in \cref{fig:radiating_bound_non_rotund}. As in part 1 by Hahn-Banach we can extend the $L_i$ to $\banachspace$, exposing faces $F_i$. Moreover by a similar construction as in part 1 we obtain functionals $L_{i,t} = L_{i} + (1-t)L_{i+1}$, $t\in(0,1)$ exposing the face $F_{i,t} = F_{i} \bigcap F_{i+1}$ which in particular contains $g_i$ and has a minimum $\overline{g_i}$. This provides us with a tangent from either $g_i$ or $\overline{g_i}$ to $g_{i+1}$ or if necessary $\overline{g_{i+1}}$ so that we can indeed get from $\hat{f}$ to $\overline{f}$ along tangents to points which are minima of a face and hence admissible. Each step includes a step away from the circle but it is clear that it can always be made arbitrarily small by varying $t$.\\
    With this process we can reach any admissible $\overline{f}$ with $\norm{\overline{f}}>\norm{\hat{f}}$, which combined with the half space bound from part 1 proves the claim.
  \end{proof}
  The proof makes clear that, just as for $\leb{1}_n$, we are only able to make statements about the minima of faces but not about their location within a face or the remaining points within the face. We thus can only obtain a result about radial symmetry in the spirit of Argyriou, Micchelli and Pontil~\citet{argyriou:09} and our previous work~\citep{schlegel:19,schlegel:19.2} by viewing $\Omega$ as a function of the faces of the norm ball as in the proof of~\cref{lma:circular_bound}. In other words we are thinking of the faces as being collapsed to one point where $\Omega$ is minimised. If we think of $\Omega$ in this way then the same intuition of almost radial symmetry as in the afore mentioned papers applies.
  \newpage
\begin{theorem}\label{thm:almost_radially_symmetric}
  A function $\Omega\,\colon\banachspace\rightarrow\R$ is admissible if and only if viewed as a function $\overline{\Omega}$ of the faces $F$ of the norm ball in $\banachspace$, $\overline{\Omega}(F) = \min_{f\in F}\Omega(f)$ it is of the form
  \[\overline{\Omega}(F) = h(\norm{f}_\banachspace \setseparator f\in F)\]
  for some nondecreasing $h\,\colon \mathopen[0,\infty\mathclose)\rightarrow\R$ whenever $\norm{f}_\banachspace\neq r$ for $r\in\mathcal{R}$. Here $\mathcal{R}$ is an at most countable set of radii where $h$ has a jump discontinuity. For any $f$ with $\norm{f}_\banachspace=r\in\mathcal{R}$ the value $\overline{\Omega}(F)$ is only constrained by the monotonicity property., i.e.\ it has to lie in between $\lim_{t\nearrow r}h(t)$ and $\lim_{t\searrow r}h(t)$.\\
  \ \\
  Moreover if a face $F$ contains an exposed point then in points of continuity of $h$ the function $\Omega$ attains its minimum in every exposed point in $F$.
\end{theorem}
\begin{proof}\textit{(Sketch)}
  It turns out that the proof of this result for uniform Banach spaces~\citep{schlegel:19} with the small adjustments for $\leb{1}_n$~\citep{schlegel:19.2} is also valid for non-reflexive Banach spaces. We are going to sketch the arguments below for completeness, more detail can be found in the afore mentioned papers.\\
  \ \\
  Firstly it is easy to show that if $\Omega$ is continuous in radial direction then $\overline{\Omega}$ has to be radially symmetric. It is clear that we can only obtain radial symmetry for admissible points but since these bound all other points from below this is sufficient. If $f$ and $g$ are admissible points of the same norm and $\Omega(f)>\Omega(g)$ say, then by \cref{lma:circular_bound} for all $1<\lambda\in\R$ we have $\Omega(\lambda g)\geq\Omega(f)$, which implies that $\modulus{\Omega(\lambda g)-\Omega(g)}\geq\modulus{\Omega(f)-\Omega(g)}>0$ contradicting radial continuity of $\Omega$.\\
  \ \\
  Moreover by the same arguments as for uniform Banach spaces and $\leb{1}_n$ we can define the radially mollified regulariser
  \[\widetilde{\Omega}(f) = \int\limits_{-1}^0 \rho(t)\Omega\left((\norm{f}-t)\frac{f}{\norm{f}}\right)\differential{t}\]
  and check by direct calculations that $\widetilde{\Omega}(f+f_T) \geq \widetilde{\Omega}(f)$ so $\widetilde{\Omega}$ is tangentially nondecreasing and hence admissible if $\Omega$ was admissible. This means that we can mollify in radial direction while preserving admissibility.\\
  \ \\
  Putting these two observations together we obtain the result. We know that $\overline{\Omega}$ is a monotone function of the norm, so a monotone function on the real line and after mollification it is in fact radially symmetric. Thus the same considerations as for uniform and reflexive Banach spaces~\citep{schlegel:19,schlegel:19.2} say that $\overline{\Omega}$ must have been of the claimed form.\\
   The converse is clear, since the value of $\overline{\Omega}$ is defined to be the minimum across each face, so minima exist and clearly satisfy the tangential bound.\\
   \ \\
   For the moreover part assume $f$ is an exposed point in a face $F$ which contains a minimum $g\neq f$ of $\Omega$. Assume further that $h$ is continuous in $\norm{f}$. Then there are tangents from $\lambda f$ to $g$ for $1-\varepsilon < \lambda < 1$. This is essentially the same situation as we saw before in \cref{fig:radially_increasing_non_rotund_minimum}\hspace*{0.2mm}, from the exposed point we can hit a point in the face above. Thus $\Omega(\lambda f) \leq \Omega(g)$. But since $g$ is a minimum for $\Omega$ and is in the same face as $f$
   \[\Omega(\lambda f) \leq \Omega(g) \leq \Omega(f)\]
   By continuity of $h$ in $\norm{f}$ we have $\Omega(\lambda f)\converges[1]{\lambda}\Omega(f)$ and so $\Omega(f) = \Omega(g)$.
\end{proof}
This shows that for any Banach space which is either strictly convex or uniformly non-rotund an admissible regulariser has to be essentially radially symmetric in the appropriate sense. This includes every space we can think of which is commonly used in applications. One should expect that similar arguments are possible for any Banach space once the space has been fixed to remove the issue of geometric variety. More precisely, if a space is relevant for an application it should be an easy check that the same proof strategy of moving between admissible points along tangents can be applied to obtain the analogous result of \cref{lma:circular_bound} and thus also of \cref{thm:almost_radially_symmetric}\hspace*{0.2mm}. This conjecture is reasonable because with $\leb{1},\leb{\infty},c_{00}$ and $\Leb{1}$ we cover some examples of spaces often thought of as ``as bad as it can get''. Many of the spaces one would think of as giving the geometric variety to make a general statement impossible can likely be seen as ``nicer'' than some of the examples covered here. Once one fixes the space it is usually not difficult to find admissible points to prove the required results.

\section{Conclusions}
The above results conclude the work by Argyriou, Micchelli and Pontil and Dinuzzo and Sch\"{o}lkopf~\citet{argyriou:09,dinuzzo:12} and our earlier work \citep{schlegel:19,schlegel:19.2}, providing a unified framework for the existence of representer theorems in general Banach spaces. Most notably this framework now includes non-reflexive Banach spaces, in particular $\leb{1}$ and $\Leb{1}$-type spaces. It thus includes common methods such as lasso~\citep{tibshirani:96} and variations of it such as square-root lasso~\citep{belloni:11}. Moreover it contains other spaces which may be very interesting for applications, but which are currently not used due to a lack in mathematical and computational theory. As an example consider $c_{0}$, the space of sequences converging to zero equipped with the maximum norm. Sequences in this space can for applications be $\varepsilon$-approximated by vectors in $c_{00}$, i.e.\ sequences of finitely many non-zero bounded coefficients. Our framework may provide a basis for the development of a theory for regularised learning in such spaces.
\subsection{Optimality}
It is clear from the proof of \cref{lma:tangential_bound} that proximinality of the subspace $Z$ is by definition the property that determines whether we can have an exact representer theorem for any given data $y_i$. We note further that \cref{def:admissibility}~\ref{item:admissible_ii} is the best we can hope for when $Z$ is not proximinal. Firstly the infimum is not always attained so we can only find a sequence of approximate minimisers. But moreover we also cannot achieve $\dist(J(f_0),\vecspan{L_i}) < \varepsilon$ for all $\varepsilon>0$ with a single $f_0\in\banachspace$.\\
To see this consider the case $\banachspace=\leb{1}$, $\banachspace^\ast=\leb{\infty}$. Let $L = {\left(\nicefrac{n}{n+1}\right)}_{n\in\N} = \left(\nicefrac12,\nicefrac23,\nicefrac34,\ldots\right)$ and consider the regularised interpolation problem
\[\min\{\Omega(f) \setseparator f\in\leb{1}, L(f) = \norm{L}_{\leb{\infty}}^2 = 1\}\]
First of all $\norm{L}_{\leb{\infty}}=1$ and there does not exist $f\in\leb{1}$ such that $\norm{f}_{\leb{1}}=1$ and $L(f) = 1$ so $\vecspan{L}\cap J(\leb{1}) = \{0\}$ and there cannot be a solution in the sense of \cref{def:admissibility}~\ref{item:admissible_i}\hspace*{0.2mm}. Furthermore any solution $f_0$ has to be of norm bigger than 1. This means that also any $\hat{L}\in J(f_0)$ would be of norm bigger than 1, $1+\delta$ for some $\delta>0$ say. But as $\hat{L}\in\leb{\infty}$ is in the image of the duality mapping, there exists an element in the sequence where the norm is attained, $\hat{L}_i = 1+\delta$. But then $\norm{\hat{L} - L} \geq \hat{L}_i - L_i > (1+\delta) - 1 = \delta > 0$ and so $f_0$ could not be a valid solution for any $\varepsilon < \delta$. This shows that the best we could hope for is finding a distinct solution for any $\varepsilon>0$.
\subsection{Future work}
Using the characterisation of admissible regularisers we showed \citep{schlegel:19,schlegel:19.2} that in fact the solution in the sense of the exact representer theorem (\cref{def:admissibility}~\ref{item:admissible_i}) is independent of the regulariser but only depends on the function space the optimisation problem is posed in. This is a very interesting result which highlights the importance of extending common learning frameworks to a variety of Banach spaces. Moreover it means that one is free to choose whichever regulariser $\Omega$ is most suitable for a given application, whether this is numerical computation or mathematical proofs.\\
The proof of this is based on Theorem 1 in Micchelli and Pontil~\citet{miccelli:04} which characterises solutions to the regularised interpolation problem as points where the distance of 0 to the subspace $\overline{f}+Z$ is attained, as discussed at the beginning of \cref{sec:representer_theorems}. It is thus plausible to expect a similar result to hold for the approximate representer theorem (\cref{def:admissibility}~\ref{item:admissible_ii}) by characterising approximate solutions as points where the distance of 0 to $\overline{f}+Z$ is almost attained.\\
\ \\
Furthermore, even when an exact representer theorem exists, in numerical implementations we are often not going to compute the exact solution but only an approximation to a given $\varepsilon$ accuracy. It would be interesting to explore whether the notion of an approximate representer theorem can lead to the design of new algorithms which may improve the computation of approximate solutions even in cases when an exact version of the theorem exists.

\begin{appendix}
\bibliographystyle{acm}
\bibliography{biblio.bib}

\section{The sandwich theorem}\label{sec:sandwich-theorem}
Using the Hahn-Banach-Lagrange theorem, a stronger version of the Hahn-Banach theorem, Simons~\citet{simons:08} proves the following Sandwich theorem.

\begin{theorem}[Sandwich Theorem]\label{thm:sandwich_thm}
  Let $V$ be a nonzero, real vector space and $P\,\colon V\rightarrow\R$ sublinear. Define a vector ordering $\leq_P$ on $V$ by
  \[u \leq_P v \mbox{ if } P(u-v) \leq 0\]
  Further assume $X$ is a nonempty set, $k\,\colon X\rightarrow(-\infty,\infty]$ not identically $\infty$ and $j\,\colon X\rightarrow V$.\\
  Suppose that for all $x_1,x_2\in\dom(k)$ there exists $u\in\dom(k)$ such that
  \[j(u) \leq_P \frac12 j(x_1) + \frac12 j(x_2) \qquad k(u) \leq \frac12 k(x_1) + \frac12 k(x_2)\]
  Then there exists a linear functional $L$ on $V$ such that $L\leq P$ and
  \[\inf\limits_{x\in X} \left[L(j(x)) + k(x)\right] = \inf\limits_{x\in X} \left[P(j(x)) + k(x)\right]\]
\end{theorem}
\ \\
Using this theorem we can easily deduce a corollary that allows us to construct a continuous linear functional of small norm which is in the subdifferential of a given convex function. For a real valued, convex function $F\,\colon V\rightarrow\R$ on a Banach space $V$ define the directional derivative of $F$ at $\overline{f}\in V$ in direction $h\in V$ as the limit
\[F'(\overline{f},h) = \lim\limits_{t\searrow 0}\frac{F(\overline{f}+th)-F(\overline{f})}{t}\]
Then $F'$ is everywhere finite and sublinear~\citep{borwein:06}. We choose $P=F'(\overline{f},\cdot)$ for some fixed $\overline{f}$ in the Sandwich theorem. For simplicity we denote the order relation by $\leq_F$. We let $X=B_V$ the unit ball in $V$ and $j(f)=f$ be the canonical embedding of $B_V$ into $V$. Lastly define $k$ to be identically 0.\\
With $j$ being the identity map we get
\[j(h) \leq_F \frac12 j(h_1) + \frac12 j(h_2) \Leftrightarrow F'(\overline{f},h-\frac12 h_1-\frac12 h_2) \leq 0\]
But for any $h_1,h_2\in B_V$ also $\nicefrac12 h_1 + \nicefrac12 h_2\in B_V$ and $F'(\overline{f},0)=0$ trivially. Further the condition on $k$ is trivially satisfied since $k$ is identically 0. Thus we obtain the following corollary of the sandwich theorem which yields a linear map in the subdifferential of $F$ at $\overline{f}$ with some control over its behaviour on the unit ball which will allow us to bound its norm.

\begin{corollary}[Sandwich theorem for subdifferentials]\label{clry:sandwich_thm_subdif}
  Let $V$ be a nonzero, real vector space, $F\,\colon V\rightarrow\R$ a convex, everywhere continuous function and $\overline{f}\in V$. Then there exists a linear functional $L$ on $V$ such that $L(\cdot)\leq F'(\overline{f},\cdot)$, i.e. $L\in\partial F(\overline{f})$, and
  \[\inf\limits_{h\in B_V} L(h) = \inf\limits_{h\in B_V} F'(\overline{f},h)\]
\end{corollary}

\section{Extension of the linear functional in the proof of \cref{lma:tangential_bound}}\label{sec:extend_functional}
In the proof of \cref{lma:tangential_bound} we obtain a functional $L\in Z^\ast$ such that $\norm{L}_{Z^\ast}<\varepsilon$. We want to extend this functional to $L\in\banachspace^\ast$ such that $L\in J(f_0^\varepsilon+f_T^\varepsilon)$. We proceed similarly to the proof of the Beurling-Livingston theorem~\citep{blazek:82,schlegel:19.2}. Let $\overline{Z}$ be the vector space generated by $Z$ and $f_0^\varepsilon$ and extend $L$ to $\overline{Z}$ by setting
    \[L(f_0^\varepsilon) = L(f_T^\varepsilon) - \norm{f_T^\varepsilon-f_0^\varepsilon}_\banachspace^2\]
    Then $L(f_T^\varepsilon-f_0^\varepsilon) = \norm{f_T^\varepsilon-f_0^\varepsilon}_\banachspace^2$ so $\norm{L}_{\overline{Z}^\ast}\geq\norm{f_T^\varepsilon-f_0^\varepsilon}$. Since the norm of $L$ on $Z$ is bounded by $\varepsilon$, and we can without loss of generality assume $\varepsilon \leq \norm{f_T^\varepsilon-f_0^\varepsilon}$, we have that the norm of $L$ on $\overline{Z}$ can only be strictly bigger than $\norm{f_T^\varepsilon-f_0^\varepsilon}$ if there is a point $\lambda f_T+\nu f_0^\varepsilon$ for $f_T\in Z$ and $\nu\neq0$ where $L$ has a value strictly bigger than $\norm{f_T^\varepsilon-f_0^\varepsilon}\cdot\norm{\lambda f_T+\nu f_0^\varepsilon}$. Since $\nu$ is nonzero we can divide through by $\nu$ and absorb the constant into the subspace $Z$ to equivalently look at points of the form $f_T+f_0^\varepsilon$. But for those points we find that
\begin{align*}
    L(f_T+f_0^\varepsilon) & = L(f_T+f_T^\varepsilon) - \norm{f_T^\varepsilon-f_0^\varepsilon}^2\\
               & \leq \varepsilon\cdot\norm{f_T+f_T^\varepsilon}-\norm{f_T^\varepsilon-f_0^\varepsilon}^2\\
               & \leq \norm{f_T^\varepsilon-f_0^\varepsilon} \cdot \modulus{\norm{f_T+f_T^\varepsilon} - \norm{f_T^\varepsilon-f_0^\varepsilon}}\\
               & \leq \norm{f_T^\varepsilon-f_0^\varepsilon} \cdot \norm{f_T+f_0^\varepsilon}
\end{align*}
Thus indeed
\[\norm{L} = \norm{f_T^\varepsilon-f_0^\varepsilon}\]
Now extend $L$ by Hahn-Banach to a linear functional on $\banachspace$ of the same norm. Then since $L(f_T^\varepsilon-f_0^\varepsilon) = \norm{f_T^\varepsilon-f_0^\varepsilon}^2$ by construction $L\in J(f_T^\varepsilon-f_0^\varepsilon)$. But then $-L\in J(f_0^\varepsilon + (-f_T^\varepsilon))$. This completes the proof.

\section{Regularisation and interpolation}\label{sec:regularisation-interpolation}
\begin{theorem}\label{thm:problem_equivalence}
  Let $\error$ be a lower semicontinuous error functional which is bounded from below. Assume further that for some $\nu\in\R^m\setminus\{0\}, y\in Y^m$ there exists a unique minimiser $0\neq a_0\in\R$ of $\min\{\error\left({(a\nu_i,y_i)}_{i\in\N_m}\right)\setseparator a\in\R\}$. Assume the regulariser $\Omega$ is  lower semicontinuous and has bounded sublevel sets.\\
  Then $\Omega$ is admissible for the regularised interpolation problem (\ref{eq:interpolation_problem}) if the pair $(\error,\Omega)$ is admissible for the regularisation problem (\ref{eq:regularisation_problem}).
\end{theorem}
The proof is very similar to the case of reflexive Banach spaces~\citep{schlegel:19.2}, which generalises the proof for Hilbert spaces given by Argyriou, Micchelli and Pontil~\citet{argyriou:09}. We are going to sketch the overall argument, which can be found in detail in the afore mentioned papers, and only go into detail where ever the proof differs for non-reflexive Banach spaces.\\
\newpage
\begin{proof}
  We are going to show that $\Omega$ is tangentially nondecreasing in the sense of \cref{lma:tangential_bound}.\\
  For every $\lambda>0$ consider the regularisation problem
  \[\min\left\{\error\left({\left(\frac{a_0}{\norm{L}^2}L(f)\nu_i,y_i\right)}_{i=1}^m\right) + \lambda\Omega(f) \setseparator f\in\banachspace\right\}\]
  Since $\ker(L)$ is proximinal~\citep[Prop.~4.7]{conway:94} we are in the situation of  \cref{def:admissibility}~\ref{item:admissible_i} and by admissibility of the pair $(\error,\Omega)$ there exist solutions $f_\lambda\in\banachspace$ such that
  \[J(f_\lambda)\cap\vecspan\{L\}\neq\emptyset\]
  Using the boundedness of sublevel sets we obtain a weakly* convergent subsequence $(f_{\lambda_l})_{l\in\N}$ such that $\lambda_l\converges{l}0$ and $f_{\lambda_l}\overset{*}{\rightharpoonup}\overline{f}^{\ast\ast}$ as $l\rightarrow\infty$. Since $\banachspace$ is not reflexive we do not get weak convergence as in the cases of Hilbert spaces and reflexive Banach spaces~\citep{argyriou:09,schlegel:19.2}.\\
  But by lower semicontinuity of $\error$ we still have that
  \[\error\left({\left(\frac{a_0}{\norm{L}^2}\overline{f}^{\ast\ast}(L)\nu_i,y_i\right)}_{i=1}^m\right) \leq \error({(a_0\nu_i,y_i)}_{i=1}^m)\]
  which as before implies that $\overline{f}^{\ast\ast}(L) = \norm{L}^2$.\\
  Just as before we obtain $\norm{\overline{f}^{\ast\ast}} = \norm{L}$ so that $\overline{f}^{\ast\ast}\in J(L)$. This means that $\overline{f}^{\ast\ast}$ and $\hat{f}$, where $\hat{f}(L) = L(f)$, both are in the same face of the norm ball in $\banachspace^{\ast\ast}$.\\
  Considering the lower semicontinuous extension $\overline{\Omega}:\banachspace^{\ast\ast}\rightarrow\R$ of $\Omega$ as before we find that $\overline{f}^{\ast\ast}$ is the minimiser of
  \[\min\{\overline{\Omega}(f^{\ast\ast}) \setseparator f^{\ast\ast}\in\banachspace^{\ast\ast}, f^{\ast\ast}(L)=\norm{L}^2\}\]
  But by Conway (\cite{conway:94}~Prop.~4.7) $\ker(L)$ is proximinal and thus by assumption the interpolation problem
  \[\min\{\Omega(f) \setseparator f\in\banachspace, L(f)=\norm{L}^2\}\]
  has a solution. When the original function attains its minimum then the minimum of the lower semicontinuous extension is not less than the minimum of the original function. Thus $\overline{\Omega}$ attains its minimum on $\hat{\banachspace}$. Thus there exists a $g\in\banachspace$ such that $\hat{g}$ is in the same face as $\overline{f}^{\ast\ast}$ and $\overline{\Omega}(\overline{f}^{\ast\ast}) = \overline{\Omega}(\hat{g})$. By the same arguments as for reflexive Banach spaces~\citep{schlegel:19.2} either $g=f$ or $f$ is an equivalent minimum or $f$ is not admissible.\\
  \ \\
  Finally note that the claim is trivially true for $L=0$ as in that case $\error$ is independent of $f$ and for every $\lambda$ the minimiser $f_\lambda$ has to be zero to satisfy $J(f_\lambda)\cap\{0\}\neq\emptyset$. This means $\Omega$ is minimised at 0.\\
\end{proof}

\begin{theorem}
Let $\error, \Omega$ be an arbitrary error functional and regulariser satisfying the general assumption that minimisers always exist. Then the pair $(\error,\Omega)$ is admissible for the regularisation problem (\ref{eq:regularisation_problem}) if $\Omega$ is admissible for the regularised interpolation problem (\ref{eq:interpolation_problem}).
\end{theorem}

\begin{proof}
  Let $f_0$ be a solution of the regularisation problem (\ref{eq:regularisation_problem}). Consider the associated regularised interpolation problem
  \[\min\{\Omega(f) \setseparator f\in\banachspace, L_i(f) = L_i(f_0) \, \forall i\in\N_m\}\]
  Since $\Omega$ is admissible for regularised interpolation, for this interpolation problem there exists a solution $\overline{f_0}$ (or $\overline{f_0^\varepsilon}$) in the sense of \cref{def:admissibility}. But then $\Omega(\overline{f_0}) \leq \Omega(f_0)$ and they have the same error as they agree on the data. Thus $\overline{f_0}$ is a solution of (\ref{eq:regularisation_problem}) in the sense of the representer theorem and the pair $(\error,\Omega)$ is admissible.\\
\end{proof}
\vspace*{-4mm}
In conclusion under the assumptions of \cref{thm:problem_equivalence} we have that the pair $(\error,\Omega)$ is admissible for the regularisation problem (\ref{eq:regularisation_problem}) if and only if $\Omega$ is admissible for the regularised interpolation problem (\ref{eq:interpolation_problem}).
\section{Proximinal subspaces}\label{sec:proximinal-subspaces}
The following corollary of Godini's theorem gives a criterium for a subspace to be proximinal which is of particular relevance to our work. Godini's theorem and the corollary, including their proofs, can be found in \cite{holmes:75}.
\begin{corollary}
  Let $V$ be a real normed vector space with unit ball $B_V$ and $W\subset V$ a closed subspace of $V$.
  \begin{enumerate}[(i)]
  \item\label{item:proximinal_fin_dim} If $W$ is finite dimensional it is proximinal.
  \item\label{item:proximinal_fin_codim} If $\codim(W)=m<\infty$ then for any basis $L_1,\ldots,L_m$ of $W^\perp$ define a map $S$ by
    \vspace*{-4mm}
    \begin{gather*}
      S\,\colon V\rightarrow\R^m\\
      S(x) = (L_1(x),\ldots,L_m(x))
    \end{gather*}
    Then $W$ is proximinal if and only if $S(B_V)$, the image of the unit ball of $V$ under the map $S$, is closed in $\R^m$.
  \end{enumerate}
\end{corollary}
Condition~\ref{item:proximinal_fin_codim} gives a condition for proximinality of the subspace $Z$ in our work, based on the linear functionals defining the regularised interpolation problem.\\
\ \\
Singer \citet{singer:70} addresses the question when every closed subspace of finite codimension, i.e.\ every possible $Z$ above, is proximinal. He proves the following result.
\begin{proposition}
  Let $\banachspace$ be a Banach space. Then all closed linear subspaces $W$ of a fixed, finite codimension $m$, where $1 \leq m \leq \dim(\banachspace)-1$ are proximinal if and only if $\banachspace$ is reflexive.
  \vspace*{-4mm}
\end{proposition}
This means that our result is optimal in the sense that for every non-reflexive Banach space $\banachspace$ there exists a combination of linear functionals $L_i$ such that $Z=\cap L_i$ is not proximinal and we cannot obtain an exact representer theorem.

\end{appendix}

\end{document}